\documentclass[10pt]{article}

\usepackage{amsmath, amssymb}
\usepackage{subfig}
\usepackage{mathdots}
\usepackage{tikz}
  \tikzstyle{dot}=[circle, inner sep=.4mm, draw=black, outer sep=2mm]
  \tikzstyle{gdot}=[circle, inner sep=.4mm, draw=black!25, outer sep=2mm]
  \tikzstyle{odot}=[circle, inner sep=.4mm, draw=black, outer sep=2mm]

\usepackage{mathpazo}
\usepackage{bm}
\usepackage{titlesec}
\titleformat{\section}
        {\Large\sc}
        {\thesection.}{1em}{}
\usepackage{amsthm}
\theoremstyle{definition}
\newtheorem{theorem}{Theorem}[section]

\newtheorem{proposition}{Proposition}
\newtheorem{definition}[theorem]{Definition}
\newtheorem{lemma}[theorem]{Lemma}
\numberwithin{equation}{section}
\numberwithin{figure}{section}

\usepackage{todonotes}

\usepackage{color}
\definecolor{lightgray}{rgb}{0.8, 0.8, 0.8}
\definecolor{darkgray}{rgb}{0.7, 0.7, 0.7}
\definecolor{darkblue}{rgb}{0, 0, .4}
\usepackage[bookmarks]{hyperref}
\hypersetup{
        colorlinks=true,
        linkcolor=darkblue,
        anchorcolor=darkblue,
        citecolor=darkblue,
        urlcolor=darkblue,
        pdfpagemode=UseThumbs,
        pdftitle={Equipopularity Classes in the Separable Permutations},
        pdfsubject={Combinatorics},
        pdfauthor={Albert, Homberger, and Pantone},
}

\usepackage{lineno}

\newcommand{\Sep}{{\mathcal{S}}}

\title{\sc Equipopularity Classes in the \\ Separable Permutations}
\author{%
Michael Albert, Cheyne Homberger, and Jay Pantone
}
\date{}

\begin{document}

\maketitle

\begin{abstract}
  When two patterns occur equally often in a set of permutations, we say that
  these patterns are equipopular. Using both structural and analytic tools, we
  classify the equipopular patterns in the set of separable permutations. In
  particular, we show that the number of equipopularity classes for length $n$ patterns
  in the separable permutations is equal to the number of partitions of $n-1$. 
\end{abstract}

\section{Introduction}

  Two sequences $\alpha_1, \alpha_2, \dots, \alpha_n$ and $\beta_1, \beta_2, \dots ,\beta_n$ are said to be \emph{order isomorphic} if they share the same relative order, i.e., $\alpha_r < \alpha_s$ if and only if $\beta_r < \beta_s$.  The permutation $\pi$ is said to \emph{contain} the permutation $\sigma$ of length $k$ as a pattern (denoted $\sigma \prec \pi$) if there is some increasing subsequence $i_1, i_2, \dots, i_k$ such that the sequence $\pi(i_1), \pi(i_2), \dots, \pi(i_k)$ is order isomorphic to $\sigma(1), \sigma(2), \dots, \sigma(k)$.  If $\pi$ does not contain $\sigma$, we say that $\pi$ \emph{avoids} $\sigma$.  For example, the permutation $\pi = 24153$ contains the pattern $\sigma = 312$, because the second, third, and fifth entries (4, 1, and  3) share the same relative order as the entries of $\sigma$.  The set of all permutations equipped with this containment order forms a partially ordered set. 

For permutations $\sigma$ and $\tau$ of lengths $n$ and $m$ respectively, the \emph{direct sum} ($\sigma \oplus \tau$) and \emph{skew sum} ($\sigma \ominus \tau$) are defined as follows : 

\[
  (\sigma \oplus \tau)(i) = 
      \begin{cases}
      \sigma(i) & 1 \leq i \leq n  \\
      \tau(i-n) + n &  n < i \leq n+m 
      \end{cases}  \\ \quad , \quad
  (\sigma \ominus \tau)(i) = 
    \begin{cases}
      \sigma(i) + m & 1 \leq i \leq n \\
      \tau(i-n)  & n < i \leq n+m
      \end{cases} 
      .
\]

These operations are more naturally understood graphically: the graph of $\sigma \oplus \tau$  (resp.,~$\sigma \ominus \tau$) is obtained by stacking the graph of $\tau$ above (resp.,~below) and to the right of that of $\sigma$. Both operations are individually, but not jointly, associative and are noncommutative. 

This paper is concerned with \emph{separable permutations} which are defined as the class of all permutations that can be formed from the length one permutation, $1$, by iterated applications of direct and skew sums.
For instance, $\pi = 543612$ is separable, as: 
\[
\pi = 543612 = ((1 \ominus 1 \ominus 1) \oplus 1)\ominus (1 \oplus 1).
\]


A \emph{permutation class} is a set of permutations that forms a downset in the pattern ordering, i.e., a set $\mathcal{C}$ for which $\pi \in \mathcal{C}$ and $\sigma \prec \pi$ implies $\sigma \in \mathcal{C}$.  The class of all permutations is denoted $\mathfrak{S}$. Enumerating and describing permutation classes has led to a variety of productive research and deep results relying on the combination of analytic, algebraic, and structural methods. Since classes are closed downward, if $\beta \not \in \mathcal{C}$ then every permutation in $\mathcal{C}$ must avoid $\beta$. In fact, any class can be defined as the set of permutations avoiding all the minimal (with respect to $\prec$) elements of its complement. These minimal elements are called the \emph{basis} of the class.
  
Since every pattern occurring in a separable permutation is itself separable, the set of all separable permutations forms a permutation class; its basis is easily shown to be $\{2413,3142\}$.  Denote the set of separable permutations by $\mathcal{S}$ and for any class, $\mathcal{C}$, denote by $\mathcal{C}_n$ the set of permutations of length $n$ within the class. 

A \emph{statistic} on a set of permutations is just a function from the set to $\mathbb{N}$. This paper is concerned primarily with the statistic of pattern \emph{occurrences}, that is, the number of times a pattern occurs in a permutation. We formalize this notion below. 

\begin{definition}
Let $\sigma$ be a (permutation) pattern, and let $\pi$ be any permutation.  Define the statistic $\nu_{\sigma}$ on $\mathfrak{S}$ by defining $\nu_\sigma(\pi)$ to be the number of times the pattern $\sigma$ occurs in the permutation $\pi$. 
\end{definition}

For example, the number of inversions of $\pi$ is just $\nu_{21}(\pi)$.

When considering a statistic it is a common first step to ask about its mean value, or equivalently, its sum. To that end, for a set $\mathcal{X}$ of permutations we define $\nu_\sigma(\mathcal{X})$ to be the total number of occurrences of $\sigma$ in the set $\mathcal{X}$, that is: 
\[
\nu_{\sigma}(\mathcal{X}) = \sum_{\pi \in \mathcal{X}} \nu_{\sigma}(\pi).
\]

In the set of all permutations, the total number of occurrences of a specified pattern depends only on the length of the pattern. To see this, let $\sigma$ and $\tau$ be two permutations of length $k$ and observe that for $n \geq k$ and any subset $X \subseteq \{1,2,\dots, n\}$ of cardinality $k$ we can define a bijection between the set of permutations whose pattern at the indices $X$ is $\sigma$ and those whose pattern is $\tau$ (namely, leave the other elements fixed and rearrange the pattern of the elements on indices $X$, retaining the set of values). So, if $\sigma$ has length $k$ and we choose $\pi$ uniformly at random in $\mathfrak{S}_n$, then the probability that the pattern of $\pi$ in any particular set of $k$ indices is equal to $\sigma$ is just $1/k!$. Linearity of expectation then implies that 
\[
\nu_{\sigma}(\mathfrak{S}_n) = \frac{n!}{k!} \binom{n}{k}.
\]

When we restrict to proper permutation classes the situation becomes less trivial, as the bijections used above are not generally available. Given a permutation class $\mathcal{C}$, we say that two patterns $\sigma$ and $\tau$ are \emph{equipopular} in $\mathcal{C}$ (denoted $\sigma \sim_{\mathcal{C}} \tau$, or just $\sigma \sim \tau$ when $\mathcal{C}$ is clear from context) if the class contains the same number of occurrences of $\sigma$ and $\tau$ at each length. That is: 
\[
\sigma \sim_{\mathcal{C}} \tau \quad \text{if and only if} \quad
\nu_\sigma(\mathcal{C}_n) = \nu_\tau(\mathcal{C}_n) \ \text{for all} \
    n \geq 1.
\]
Obviously, $\sim_{\mathcal{C}}$ is an equivalence relation. All permutations not belonging to $\mathcal{C}$ are equipopular (since the number of occurrences of any such pattern in permutations of $\mathcal{C}$ is 0) and thus we generally restrict its domain to $\mathcal{C}$. We refer to the equivalence classes of $\mathcal{C}$ as \emph{equipopularity classes}. Note that all permutations belonging to an equipopularity class are of the same length, since the least $n$ for which $\nu_{\sigma} (\mathcal{C}_n) > 0$ is $n=|\sigma|$.
  
Motivated by a question of Cooper, B\'ona~\cite{bona:absence, bona:surprising} showed that, in the class of permutations avoiding the pattern $132$, the patterns $213, 231$, and $312$ are equipopular.  Homberger~\cite{homberger:expected} extended this symmetry to show that the number of occurrences of the pattern $231$ is identical among the permutations avoiding $123$ and those avoiding $132$. Rudolph~\cite{rudolph:popularity} presented a surjection from integer partitions to equipopularity classes, and Chua and Sankar~\cite{chua:popularity} showed that this map is actually a bijection. The class of $132$ avoiding permutations is a proper subset of the separable permutations studied here. 

  This paper enumerates and classifies the equipopularity classes of the separable permutations. In Section~\ref{sec:equi} we use the structural decomposition of the separables to build a surjection from integer partitions classes onto equipopularity classes, generalizing the results of Rudolph. In Section~\ref{sec:nonequi} we show that this surjection is a bijection, generalizing the results of Chua and Sankar and proving that there are exactly partition-many equipopularity classes in the separable permutations. 




\section{The Separable Permutations}
  
The separable permutations exhibit numerous symmetries, and their recursive structure will be a fundamental tool used in this investigation.  Since every permutation $\pi \in \mathcal{S}$ can be written as a sequence of direct and skew sums of the permutation $1$, it follows that $\pi$ has a (not necessarily unique) decomposition into smaller separable permutations $\sigma$ and $\tau$ in one of two different ways:
\[
\pi = \sigma \oplus \tau \quad \text{or} \quad \pi = \sigma \ominus \tau.
\]

\begin{definition}
A permutation is \emph{sum decomposable} if it can be written as a direct sum of two non-trivial permutations, and \emph{sum indecomposable} otherwise. Skew (in)decomposability is defined analogously. 
\end{definition}

Since it is clearly impossible for a permutation to be both sum and skew decomposable, it follows then that every sum decomposable separable permutation can be written uniquely as a direct sum of its sum indecomposable parts. That is, for every sum decomposable permutation $\pi$, there exists a unique integer $k \geq 2$ and a unique sequence of sum indecomposable permutations $\sigma_i$, for $1 \leq i \leq k$ such that 
\[
\pi = \sigma_1 \oplus \sigma_2 \oplus \dots \oplus \sigma_k.
\] 
A similar decomposition holds for skew decomposable permutations. 

The only separable permutation that is both sum and skew indecomposable is $1$, so given a separable permutation $\pi$ which is, say, sum decomposable, its sum indecomposable summands, if not $1$, must be skew sums of skew indecomposable permutations. This provides a recursive decomposition of $\pi$, and this decomposition can be thought of as being represented by a particular sort of tree.

\begin{definition}
A decomposition tree is a rooted plane tree in which each non-leaf node is labelled with either $\oplus$ or $\ominus$, and has at least two children.  The non-leaf children of a node labelled $\oplus$ are labelled with $\ominus$, and vice versa. 
\end{definition}

There is then a bijection $\Gamma$ between $\mathcal{S}$ and the set of decomposition trees defined as follows:
\begin{itemize}
\item
$\Gamma(1)$ is the tree consisting of a single node;
\item
if $\pi \in \Sep$ is sum-decomposable, and $\pi = \sigma_1 \oplus \sigma_2 \oplus \dots \oplus \sigma_k$ where each $\sigma_i$ is sum indecomposable, then $\Gamma(\pi)$ has a root node labelled $\oplus$, having $k$ children and the subtrees rooted at its children are $\Gamma(\sigma_1), \Gamma(\sigma_2), \dots, \Gamma(\sigma_k)$ in that order;
\item
if $\pi$ is skew decomposable, $\Gamma(\pi)$ is defined similarly to the preceding case.
\end{itemize}

For example, the tree $\Gamma(\pi)$ for $\pi = 215643798$ is shown in  Figure~\ref{fig:decomp-tree}, and the full decomposition of $\pi$ is as follows:

 \[
 \begin{aligned} 
    215643798 &= (21) \oplus (3421) \oplus 1 \oplus (21) \\
            &= \big( 1 \ominus 1 \big) 
              \oplus \big(12 \ominus 1 \ominus 1\big) 
              \oplus 1 \oplus \big( 1 \ominus 1\big) \\
            &= \big( 1 \ominus 1 \big) 
              \oplus \big((1 \oplus 1) \ominus 1 \ominus 1\big) 
              \oplus 1 \oplus \big( 1 \ominus 1\big).
\end{aligned}
\]

  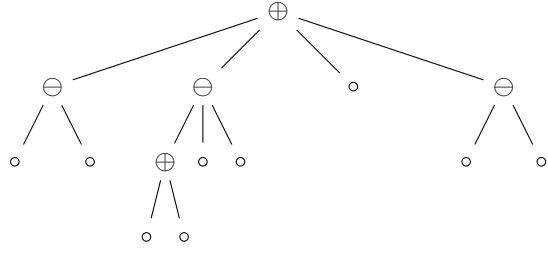
\begin{figure}[t]
    \centering
    \begin{tikzpicture}
      \node[] (a1) at (6,5) {$\oplus$};

      \node[] (b1) at (3,4) {$\ominus$};
      \node[] (b2) at (5,4) {$\ominus$};
      \node[dot] (b3) at (7,4) {};
      \node[] (b4) at (9,4) {$\ominus$};

      \node[dot] (c1) at (2.5,3) {};
      \node[dot] (c2) at (3.5,3) {};

      \node (c3) at (4.5,3) {$\oplus$};
      \node[dot] (c4) at (5,3) {};
      \node[dot] (c5) at (5.5,3) {};

      \node[dot] (c6) at (8.5,3) {};
      \node[dot] (c7) at (9.5,3) {};

      \node[dot] (d4) at (4.25,2) {};
      \node[dot] (d5) at (4.75,2) {};

      \draw (a1) -- (b1);
      \draw (a1) -- (b2);
      \draw (a1) -- (b3);
      \draw (a1) -- (b4);

      \draw (b1) -- (c1);
      \draw (b1) -- (c2);

      \draw (b2) -- (c3);
      \draw (b2) -- (c4);
      \draw (b2) -- (c5);

      \draw (b4) -- (c6);
      \draw (b4) -- (c7);

      \draw (c3) -- (d4);
      \draw (c3) -- (d5);
    \end{tikzpicture}
  \caption{The decomposition tree $\Gamma(\pi)$ of $\pi = 215643798$.}
  \label{fig:decomp-tree}
  \end{figure}

These tree representations will be used in Section~\ref{sec:equi} to investigate and describe the relationships between equipopular patterns. This recursive structure will be used to develop functional relationships between the generating functions for pattern popularity, especially in Section~\ref{sec:nonequi}. As an illustrative example, we rederive the enumeration of the separable permutations. 

\begin{theorem}[\cite{shapiro:bootstrap-perco, west:generating-tree:95}]
The number of separable permutations of length $n$ is equal to the $n^{\text{th}}$ (large) Schr\"oder number. 
\end{theorem}

\begin{proof}
Let $s_n$ be the number of separable permutations of length $n$ (with $s_0 = 1$), and let $$ S = \sum_{n \geq 0} s_n t^n. $$ Let $S^\oplus$ and $S^\ominus$ be the generating functions for the sum and skew decomposable separable permutations, respectively. Since every separable permutation of length at least two is either sum or skew decomposable, and since no permutation can be both, we have that 

\begin{equation} \label{eqn:sep1}
      S = S^\oplus + S^\ominus + 1 + t.
\end{equation}

Every sum decomposable separable permutation $\pi$ can be written as $\sigma \oplus \tau$, where $\tau$ is separable and $\sigma$ is sum indecomposable (and hence either skew decomposable or has length one).  This leads to the following relationship:

    \begin{equation} \label{eqn:sep2}
      S^\oplus = (S^\ominus + t) (S-1). 
    \end{equation}

    Finally, the reverse of a sum decomposable permutation is skew decomposable and vice versa, which implies:

    \begin{equation} \label{eqn:sep3}
      S^\oplus = S^\ominus.
    \end{equation}

    Combining equations~\ref{eqn:sep1}-\ref{eqn:sep3} leads to a functional equation that can then be solved algebraically to obtain the generating function for the Schr\"oder numbers:

    $$ \begin{aligned}
      S &= \frac{ 3 - t - \sqrt{1 - 6t - t^2}}{2}. 
      \end{aligned} $$
  \end{proof}

\section{Equipopular Patterns}
\label{sec:equi}

In this section we show that the number of equipopularity classes for patterns of length $n$ is at most the number of partitions of the integer $n$. By the the decomposition of separable patterns, we determine sufficient conditions for two patterns to be equipopular (these conditions will be shown to be necessary in the next section). Our primary tool will be the decomposition trees introduced in the previous section. That is, we think of separable permutations and their decomposition trees interchangeably.

Recall that two patterns of a different length can never be equipopular, and that we are only interested in equipopularity for separable patterns (since no other patterns occur at all). Therefore, within the set of all decomposition trees of a fixed size, our goal is to classify those which correspond to equipopular patterns. This classification is obtained by describing structural transformations on trees which lead to equipopular patterns. 

\subsection{Preserving Equipopularity}
  
For a specified pattern $\sigma$, define a \emph{$\sigma$-marked permutation} to be a permutation $\pi$ for which a subset of entries forming a $\sigma$ pattern have been marked. Note that, for a set $\mathcal{X}$ of permutations, the total number of occurrences of a pattern $\sigma$ in $\mathcal{X}$, i.e., $\nu_{\sigma}(\mathcal{X})$, is equal to the number of distinct $\sigma$-marked permutations $\pi'$, where the unmarked permutation $\pi$ lies in $\mathcal{X}$.

\begin{definition}
For a permutation $\pi \in \mathfrak{S}_n$, the \emph{reverse, complement, and inverse} of $\pi$, denoted $\pi^{r}, \pi^{c}$, and $\pi^{-1}$ respectively, are defined as follows:
\[
\begin{aligned}
      (\pi^{r})(i) &= \pi(n-i + 1), \\
      (\pi^{c})(i) &= n - \pi(n) + 1, \text{ and} \\
      (\pi^{-1})(\pi(i)) &= i.
\end{aligned}
\]
\end{definition}

These three operations respect pattern containment, in the sense that, for each $i \in \{r,c,-1\}$, we have 

\begin{equation} \label{eqn:symmetry}
\sigma \prec \pi \quad \text{if and only if} \quad \sigma^i \prec \pi^i. 
\end{equation}
  
These three symmetries generate a group of automorphisms of $(\mathfrak{S}, \prec)$ isomorphic to the dihedral group of order eight, and Smith~\cite{smith:permutation-rec:} showed that this is the full group of automorphisms of the pattern poset. The \emph{symmetries} of a permutation $\pi$ are the images of $\pi$ under the action of this group. 

Note that if $\pi$ is separable, then the symmetries of $\pi$ are also separable. Further, if we apply one of these symmetries $f$ to a $\sigma$-marked permutation, we obtain an $f(\sigma)$-marked permutation. It follows then that if two patterns are related by a symmetry, then they are equipopular.  Considering separable permutations as their decomposition trees, each symmetry corresponds to a tree transformation: complementation equates to inverting the signs at each internal node, reversal to inverting the signs and reversing the order of the children of every node, and inversion reverses the order of the children of each $\ominus$-node. 

Pattern containment can be reconsidered entirely in terms of trees, though the analogous relationship is a bit more technical and requires an additional pair of definitions. 

\begin{definition}
For a decomposition tree $T$ and a set $L$ of leaves, the \emph{least common parent} of $L$ is the unique node which is furthest from the root and lies above each leaf in $L$. The \emph{skeleton} of $L$ is the tree whose vertices are the least common parents of each subset of $L$, and the ordering is induced from $T$ by transitivity.  The \emph{reduced skeleton} of $L$ is obtained by contracting any two same-sign adjacent nodes of the skeleton into a single node (with the same sign). 
\end{definition}

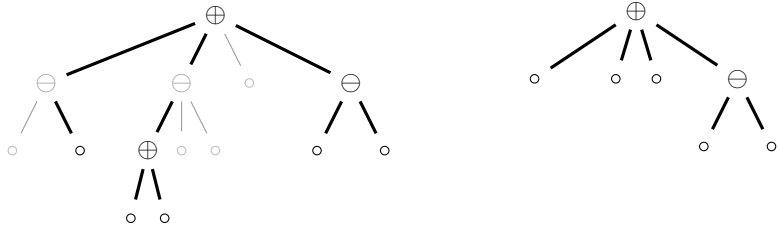
\begin{figure}
    \centering
    \begin{tikzpicture}[scale=.9]
      \node[] (a1) at (6,5) {$\oplus$};

      \node[] (b1) at (3.5,4) {\color{black!40}{$\ominus$}};
      \node[] (b2) at (5.5,4) {\color{black!40}{$\ominus$}};
      \node[gdot] (b3) at (6.5,4) {};
      \node[] (b4) at (8,4) {$\ominus$};

      \node[gdot] (c1) at (3,3) {};
      \node[dot] (c2) at (4,3) {};

      \node (c3) at (5,3) {$\oplus$};
      \node[gdot] (c4) at (5.5,3) {};
      \node[gdot] (c5) at (6,3) {};

      \node[dot] (c6) at (7.5,3) {};
      \node[dot] (c7) at (8.5,3) {};

      \node[dot] (d4) at (4.75,2) {};
      \node[dot] (d5) at (5.25,2) {};

      \draw[draw=black,very thick] (a1) -- (b1);
      \draw[draw=black, very thick] (a1) -- (b2);
      \draw[draw=black!40] (a1) -- (b3);
      \draw[draw=black, very thick] (a1) -- (b4);

      \draw[draw=black!40] (b1) -- (c1);
      \draw[draw=black, very thick] (b1) -- (c2);

      \draw[draw=black, very thick] (b2) -- (c3);
      \draw[draw=black!40] (b2) -- (c4);
      \draw[draw=black!40] (b2) -- (c5);

      \draw[draw=black, very thick] (b4) -- (c6);
      \draw[draw=black, very thick] (b4) -- (c7);

      \draw[draw=black, very thick] (c3) -- (d4);
      \draw[draw=black, very thick] (c3) -- (d5);
    \end{tikzpicture}
    \hspace{4pc}
    \begin{tikzpicture}[scale=.9]
      \node[] (a1) at (6,5) {$\oplus$};
      \node[dot] (b1) at (4.5,4) {};
      \node[dot] (b2) at (5.7,4) {};
      \node[dot] (b3) at (6.3,4) {};
      \node[] (b4) at (7.5,4) {$\ominus$};
      \node[dot] (c1) at (7,3) {};
      \node[dot] (c2) at (8,3) {};
      \node at (6,2) {};

      \draw[very thick] (a1) -- (b1);
      \draw[very thick] (a1) -- (b2);
      \draw[very thick] (a1) -- (b3);
      \draw[very thick] (a1) -- (b4);
      \draw[very thick] (b4) -- (c1);
      \draw[very thick] (b4) -- (c2);

\end{tikzpicture}
\caption{The skeleton and reduced skeleton of an occurrence of the pattern $12354$ in the permutation $215643798$.}
\label{fig:skeleton}
\end{figure}

An example of the skeleton and reduced skeleton of a set of leaves is shown in Figure~\ref{fig:skeleton}. 

\begin{lemma}
Let $\sigma, \pi \in \mathcal{S}$. Then $\sigma \prec \pi$ if and only if there is some subset $L$ of leaves of $\Gamma(\pi)$ for which the reduced skeleton of $L$ is equal to $\Gamma(\sigma)$. Further, the position of the leaves of $L$ within $\Gamma(\pi)$ correspond to the position of the occurrence of $\sigma$ within $\pi$. 
\end{lemma}

\begin{proof}
This follows from the fact that the positional relationship between any pair of entries of $\pi$ is determined by the left-to-right order of the corresponding leaves of $\Gamma(\pi)$ while the value relationship is determined by the sign of the least common parent of those leaves. That is, if $i < j$ then the leaf corresponding to $\pi_i$ is to the left of the leaf corresponding to $\pi_j$, and if $\pi_i < \pi_j$ then their least common parent is an $\oplus$-node, while if $\pi_i > \pi_j$ their least common parent is an $\ominus$-node. Further, in the reduced skeleton of $L$, while the least common parent of a pair of leaves may correspond to a contraction of their least common parent in $\Gamma(\pi)$ with some ancestral node, all nodes contracted into a single node share a common label.
\end{proof}

Define a \emph{$\sigma$-marked decomposition tree} to be a tree with marked leaves, such that the reduced skeleton of the marked leaves is equal to $\Gamma(\sigma)$. By the first observation of this section about the correspondence between separable permutations and their decomposition trees and the result above it follows that the number of occurrences of $\sigma$ in separable permutations of length $n$ is equal to the number of distinct $\sigma$-marked trees with $n$ leaves. When we have a $\sigma$-marked tree we will often consider the nodes belonging to the skeleton of the marked leaves to be marked as well.

We now seek to describe some transformations of decomposition trees that preserve equipopularity. Since the labelling of a decomposition tree is uniquely determined by the label of its root, we permit some transformations that might require changing labels in some subtrees to ensure the alternation of $\oplus$ and $\ominus$ labels on internal nodes.

\begin{proposition} \label{prop:node_switch}
Let $\sigma$ be a separable permutation and let $v$ and $w$ be nodes of $\Gamma(\sigma)$ such that neither is an ancestor of the other. Let $\Gamma(\tau)$ be the decomposition tree obtained by exchanging the subtrees rooted at $v$ and $w$ in $\Gamma(\sigma)$. Then $\sigma$ and $\tau$ are equipopular.
\end{proposition}

\begin{proof}
We describe a bijection between $\sigma$-marked trees and $\tau$-marked trees which preserves the number of leaves.  Note that by the ancestry condition, both $v$ and $w$ must be internal nodes or leaves, and so have parents $v'$ and $w'$ respectively.
    
Let $T$ be any $\sigma$-marked tree. In the marked skeleton of $T$, $v'$ may correspond to a set of nodes, $V'$. However, there is a node in $V'$ that has a child $c_v$ such that all the marked leaves below $c_v$ correspond to leaves of the subtree rooted at $v$ (and all such leaves occur below $c_v$). Correspondingly, among the set of nodes $W'$ corresponding to $w'$ there is one that has a child $c_w$ such that all the marked leaves below $c_w$ correspond to leaves of the subtree rooted at $w$. Form a new marked decomposition tree by exchanging the subtrees rooted at $c_v$ and $c_w$ in $T$. This new tree is $\tau$-marked. This process is clearly invertible and preserves the number of leaves. Therefore, $\sigma$ and $\tau$ are equipopular. 

See Figure~\ref{fig:shuffle} for an example. 
\end{proof}


Using just this proposition, we can repeatedly exchange the rightmost leaf of a tree with any internal node not in the rightmost branch producing, from some arbitrary initial separable permutation $\sigma$, a chain of equipopular permutations and culminating in one, $\tau$ where all but the rightmost child of any node are leaves. We will see analytically below that any two such $\tau$ are equipopular provided that the multisets of degrees of their internal vertices are equal. However, the following result also provides an alternative bijective proof of this fact -- thereby allowing in principal the construction of bijections between $\sigma$-marked and $\tau$-marked permutations whenever $\sigma \sim \tau$ in the separable permutations.

  \begin{figure}[t]
    \centering
    \subfloat[The second pattern can be obtained from the first by
    interchanging the first two children of the root.]{
    \begin{tikzpicture}[scale=.4]
    \node[label=east:{}] (A) at (5,10) {$\oplus$};
    \node[dot] (b1) at (3,8) {};
    \node (b2) at (5,8) {$\ominus$};
    \node (b3) at (7,8) {$\ominus$};
    \node[dot] (c1) at (4.5,6) {};
    \node[dot] (c2) at (5.5,6) {};
    \node[dot] (c3) at (6.5,6) {};
    \node[dot] (c4) at (7,6) {};
    \node[dot] (c5) at (7.5,6) {};
    \foreach \to/\from in 
      {A/b1, A/b2, A/b3, b2/c1, b2/c2, b3/c3, b3/c4, b3/c5}
    \draw (\to) -- (\from);
    \end{tikzpicture}
    \hspace{2pc}
    \begin{tikzpicture}[scale=.5]
    \node[label=east:{}] (A) at (5,10) {$\oplus$};
    \node[dot] (b1) at (5,8) {};
    \node (b2) at (3,8) {$\ominus$};
    \node (b3) at (7,8) {$\ominus$};
    \node[dot] (c1) at (2.5,6) {};
    \node[dot] (c2) at (3.5,6) {};
    \node[dot] (c3) at (6.5,6) {};
    \node[dot] (c4) at (7,6) {};
    \node[dot] (c5) at (7.5,6) {};
    \foreach \to/\from in 
      {A/b1, A/b2, A/b3, b2/c1, b2/c2, b3/c3, b3/c4, b3/c5}
    \draw (\to) -- (\from);
    \end{tikzpicture}
    } \hfill
    \subfloat[Any tree with the first pattern marked can be transformed into one
    with the second pattern marked. The set of nodes in the dotted box
    correspond to the root node of the pattern.]{
    \begin{tikzpicture}[scale=.4]
    \node (A1) at (5,10) {$\oplus$};
    \node (A2) at (7,8) {\color{black!25}{$\ominus$}};
    \node (A3) at (9,6) {$\oplus$};
    \node[dot] (d1) at (3,8) {};
    \node[gdot] (d2) at (5.5,6) {};
    \node (B1) at (8,4) {$\ominus$};
    \node (B2) at (10,4) {$\ominus$};
    \node[dot] (c1) at (7.5,2) {};
    \node[dot] (c2) at (8.5,2) {};
    \node[dot] (c3) at (9.5,2) {};
    \node[dot] (c4) at (10,2) {};
    \node[dot] (c5) at (10.5,2) {};
    \foreach \to/\from in {A1/d1, A1/A2, A2/A3, A3/B1, A3/B2, 
          B1/c1, B1/c2, B2/c3, B2/c4, B2/c5}
      \draw[thick] (\to) -- (\from);
    \draw[black!25] (A2) -- (d2);
    \draw[rounded corners=1mm, dotted] (A1.north) -- (A1.west) -- (A3.south) --
    (A3.east) -- cycle;
    \end{tikzpicture}
    \hspace{2pc}
    \begin{tikzpicture}[scale=.5]
    \node (A1) at (5,10) {$\oplus$};
    \node (A2) at (7,8) {\color{black!25}{$\ominus$}};
    \node (A3) at (9,6) {$\oplus$};
    \node[dot] (d1) at (8,4) {};
    \node[gdot] (d2) at (5.5,6) {};
    \node (B1) at (3,8) {$\ominus$};
    \node (B2) at (10,4) {$\ominus$};
    \node[dot] (c3) at (9.5,2) {};
    \node[dot] (c4) at (10,2) {};
    \node[dot] (c5) at (10.5,2) {};
    \node[dot] (c1) at (2.5,6) {};
    \node[dot] (c2) at (3.5,6) {};
    \foreach \to/\from in {A3/d1, A1/A2, A2/A3, A1/B1, A3/B2, 
          B1/c1, B1/c2, B2/c3, B2/c4, B2/c5}
      \draw[thick] (\to) -- (\from);
    \draw[black!25] (A2) -- (d2);

    \draw[rounded corners=1mm, dotted] (A1.north) -- (A1.west) -- (A3.south) --
    (A3.east) -- cycle;
    \end{tikzpicture}
    }
\caption{Interchanging of children of a node within a pattern can be extended
to the transformation of a marked pattern within a larger tree.}
  \label{fig:shuffle}
  \end{figure}
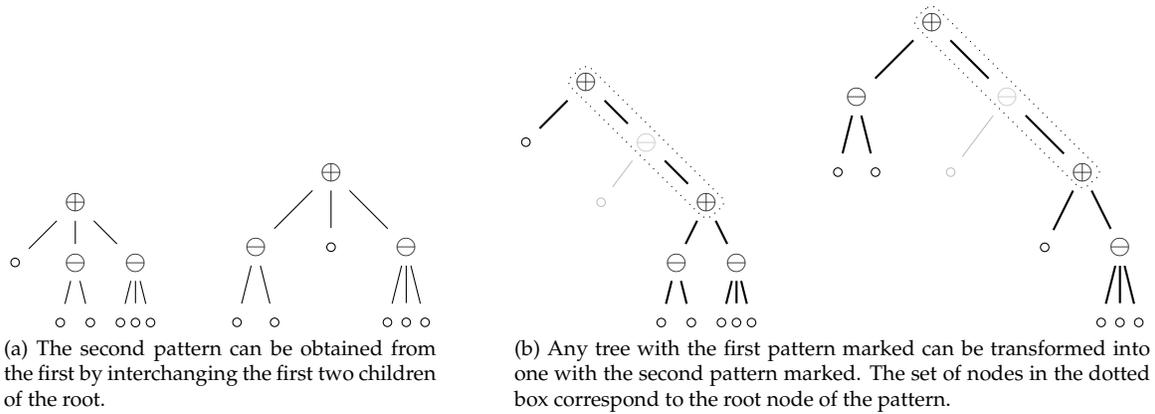

\begin{definition}
Let $a$ be an internal node of a decomposition tree $T$, with rightmost child $b$ that is also internal. Let $c$ be the rightmost child of $b$. Let $T_a$, $T_b$, and $T_c$ denote the subtrees rooted at $a$, $b$ and $c$ respectively. Form the tree $T_{ab}$ by attaching $T_c$ to $T_a \setminus T_b$ at the position $b$ occupied, then attaching this to $T_b \setminus T_c$ at the position $c$ occupied, and finally attaching this to $T \setminus T_a$ at the position $a$ occupied. Then we will call $T_{ab}$ a \emph{forest exchange} of $T$.
\end{definition}

An alternative description of $T_{ab}$ (and one that accounts for the name) is that it is obtained from $T$ by exchanging all but the rightmost tree of the forests below $a$ and $b$ respectively. See Figure \ref{fig:rotation} for illustration.

\begin{figure}
		\minipage{0.45\textwidth}
			\begin{center}
				 \begin{tikzpicture}[scale=.745]
                                        \node (top) at (2.5,5) {};
                                        \node[circle, draw=black, inner sep=.5mm, outer sep=2mm, label=north east:{$a$}] (a) at (3,4) {};
                                        \node[circle, draw=black, inner sep=.5mm, outer sep=2mm, label=north east:{$b$}] (b) at (4,2) {};
                                        \node[circle, draw=black, inner sep=.5mm, outer sep=2mm, label=north east:{$c$}] (c) at (5,0) {};
                                        \node (bot) at (5.5,-1) {};
                                        
                                        \node[outer sep=4mm] (F) at (2.6,2.5) {$F$};
                                        \node[outer sep=4mm] (G) at (3.6,0.5) {$G$};
                                        
                                        \draw (a) -- (b);
                                        \draw (b) -- (c);
                                        
                                        \draw[dotted] (a) -- (top);
                                        \draw[dotted] (c) -- (bot);
                                        
                                        \draw[black!25] (a) -- ++(.3, -2) -- ++(-1.7,0) -- (a);
                                        \draw[black!25] (b) -- ++(.3, -2) -- ++(-1.7,0) -- (b);
                                        
                                        \draw[<->] (F) --  (G);
                                    \end{tikzpicture}
			\end{center}
			\caption{When exchanging across $ab$, the forests denoted by $F$ and $G$ are interchanged. }
			\label{fig:rotation}
		\endminipage\hfill
		\minipage{0.45\textwidth}
			\begin{center}
				\begin{tikzpicture}[scale=.5]
                                            \node[circle, inner sep=.5mm, outer sep=1mm] (1) at (5,10) {$\oplus$};
                                            \node[circle, inner sep=.5mm, outer sep=1mm] (21) at (6,8) {$\ominus$};
                                            \node[label=below:{$\lambda_1$}] (22) at (4,8) {$\dots$};
                                        
                                            \node[circle, inner sep=.5mm, outer sep=1mm] (31) at (7,6) {$\oplus$};
                                            \node[label=below:{$\lambda_2$}] (32) at (5,6) {$\dots$};
                                        
                                            \node[circle, inner sep=.5mm, outer sep=1mm] (41) at (9,2) {$\circ$};
                                            \node[label=below:{$\lambda_{k-1}$}] (42) at (7,2) {$\dots$};
                                        
                                            \node[label=below:{$\lambda_{k} + 1$}] (5) at (9,0) {$\dots$};
                                            \node[outer sep = 3mm] (inv) at (8,4) {};
                                        
                                            \draw[thick, dotted] (31) -- (inv);
                                            \draw[thick, dotted] (41) -- (inv);
                                            \draw[thick, dotted] (42) -- (inv);
                                        
                                            \foreach \to/\from in {1/21, 1/22, 21/31, 21/32, 41/5}
                                              \draw (\to) -- (\from);
                                \end{tikzpicture}
			\end{center}
			\caption{The tree $t(\lambda)$, where $\lambda = \lambda_1, \lambda_2, \dots, \lambda_k$. The numbers indicate the number of leaf children, and the sign of the lowest internal node is determined by the parity of its height.}
			\label{fig:talltree}
		\endminipage
	\end{figure}

\begin{proposition}\label{prop:abexchange}
Let $\sigma$ be a separable pattern, and let $a$ and $b$ be internal nodes of $\Gamma(\sigma)$ for which a forest-exchange is possible. Let $\tau$ be the separable pattern such that $\Gamma(\tau)$ is equal to the tree resulting from the forest-exchange. Then $\sigma$ and $\tau$ are equipopular.
\end{proposition}
  
\begin{proof}
The proof is analogous to that of the preceding proposition. Let a $\sigma$-marked tree $T$ be given. Among the nodes $A$ of $T$ which correspond to node $a$ of $\Gamma(\sigma)$ there is an eldest $a'$. Similarly define nodes $b'$ and $c'$. Now perform the analog of the $ab$-exchange in $T$, but attaching, e.g., $T_{c'}$ to $T_{a'} \setminus T_{b'}$ as the rightmost child of the rightmost vertex in $A$ (and similarly for the other parts of the exchange). The resulting tree is  $\tau$-marked and the operation is invertible, proving the result.
\end{proof}

%

\subsection{Identifying Equipopular Permutations}

The remainder of this section develops a method of identifying the equipopularity classes, and constructing a canonical representative from each.  We first build a correspondence between partitions and classes. 

\begin{definition}
Let $T$ be a decomposition tree, and let $i_1, i_2, \dots, i_r$ be the internal nodes of $T$.  The \emph{signature} of $T$ is the multiset $\{d(i_j) - 1\}_{j=1}^r$, where $d(i_j)$ denotes the down degree (the number of immediate children) of the node $i_j$. 
\end{definition}
  
For every partition $\lambda$ of an integer $n$, there exists a decomposition tree with $n+1$ leaves having signature $\lambda$, as shown in Figure~\ref{fig:talltree}. Denote this tree by $T(\lambda)$, and let $\omega(\lambda)$ be the permutation for which $\Gamma(\omega(\lambda)) = T(\lambda)$. 

This family of trees provides a canonical element for each equipopularity class. Indeed, the following theorem shows that every decomposition tree can be transformed into such a $\lambda$-tree through the use of the equipopularity-preserving operations introduced in the previous section. 

\begin{theorem}
Let $\sigma$ be a permutation, and suppose that $\Gamma(\sigma)$ has signature $\lambda$. Then $\sigma$ is equipopular to $\omega(\lambda)$. 
\end{theorem}

\begin{proof}
As noted previously, we can first use Proposition \ref{prop:node_switch} to find a tree equipopular to $\Gamma(\sigma)$ with the property that only the rightmost child of any node can be internal. The resulting tree is similar to $t(\lambda)$ except that the order of the degrees of the internal nodes may differ (the multiset of degrees is the same). But now we can arbitrarily permute these internal nodes (since they all lie in the rightmost branch) using a series of forest-exchanges and Proposition \ref{prop:abexchange}.
  \end{proof}
  
Translating back into the language of permutations, we have shown that every separable permutation of length $n$ is equipopular to one in a particular family, and that this family is in bijection with the set of partitions of the integer $n-1$. For a partition $\lambda = \lambda_1, \lambda_2, \dots, \lambda_k$, using the definition of the decomposition trees and letting $I(n)$ denote the identity permutation of length $n$, we have 

\[
\omega(\lambda) = 
\begin{cases}
I(\lambda_1) \oplus \omega(\lambda_2, \dots, \lambda_k)^c & \text{if $k > 1$} \\
I(\lambda_1+1) & \text{if $k = 1$}
\end{cases}
\]

The permutations $w(\lambda)$ are referred to as the \emph{wedge permutations}, due to the shape of their plots. See Figure~\ref{fig:wedge} for an illustration. 
  
\begin{figure}
  \centering
  \begin{tikzpicture}[scale=.4]
    \draw[very thick, dotted] (0,0)--(1,1) node[pos=.5, above left] {$\lambda_1$};
    \draw[very thick, dotted] (1.5,9)--(2.5,8) node[pos=.5, below left] {$\lambda_2$};
    \draw[very thick, dotted] (3,1.5)--(4,2.5) node[pos=.5, above left] {$\lambda_3$};
    \draw[very thick, dotted] (4.5,7.5)--(5.5,6.5) node[pos=.5, below left] {$\lambda_4$};
    \draw[very thick, dotted] (6,3)--(7,4) node[pos=.5, above left] {$\lambda_5$};

    \node at (8,5) {$\dots$};

  \end{tikzpicture}
  \caption{The plot of the wedge permutation $w(\lambda)$, where $\lambda = \lambda_1, \lambda_2, \dots, \lambda_k$. Recall that the size of the final monotone segment is $\lambda_k + 1$.} \label{fig:wedge}
\end{figure}
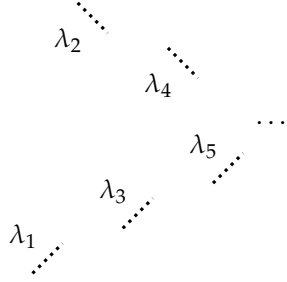

\section{Non-Equipopular Patterns}
\label{sec:nonequi}

In the previous section we proved that if two patterns correspond to the same integer partition, they are equipopular. Here we show that if two patterns correspond to different partitions, they do in fact have different enumerations, showing that the equipopularity classes are in bijection with integer partitions. The methods used here contrast sharply with those used previously: we utilize a variety of analytic techniques to prove this result. 

This section is split into two parts. Given the popularity generating function of a wedge permutation, we first show that we can factor it into generating functions for the popularity of the monotone runs in the pattern (along with some other terms). Next, we show that given the product of several such generating functions, we can identify the lengths of the monotone runs. Together this will show that each popularity generating function is unique to a specific wedge permutation.

\subsection{Factoring Popularity Generating Functions}

For $\sigma \in \mathcal{S}$, let $P_\sigma(t)$ be the generating function for the popularity of the pattern $\sigma$. That is, let 
\[
P_\sigma(t) = \sum_{n \geq 0} \nu_\sigma\left(\mathcal{S}_n\right)t^n.
\] 
Recall that we use $I(n)$ to denote the pattern $1 2 \cdots n$. 

%

\begin{lemma}\label{lem:factorization}
There exists a family of functions $\{F_m\}_{m=0}^\infty$ such that if $\pi$ is an arbitrary sum indecomposable separable permutation, and $m \geq 0$. Then $P_{I(m) \oplus \pi} = F_m P_\pi$. 
\end{lemma}

\begin{proof}
Let $P_{I(m) \oplus \pi}^\oplus$ and $P_{I(m) \oplus \pi}^\ominus$ denote the popularity generating functions of the pattern $I(m) \oplus \pi$ within the sum and skew decomposable separable permutations, and let $P_{I(m) \oplus \pi}^{\oplus, \bullet}$ denote the popularity generating function within skew \emph{indecomposable} permutations. Let $S$ denote the generating function for the separable permutations including the empty permutation, i.e., with constant term one. 

We proceed by induction on $m$, noting that for the base case $m = 0$ the result is trivial ($F_0 = 1$).  

Suppose that $m \geq 1$. Since the length of $I(m) \oplus \pi$ is at least two, the permutation of length one does not contain $I(m) \oplus \pi$. Then, since every permutation of $S$ of length at least two is either sum or skew decomposable, we have

\begin{equation} \label{eqn:factorization1}
      P_{I(m) \oplus \pi} = 
      P_{I(m) \oplus \pi}^\oplus + P_{I(m) \oplus \pi}^\ominus.
\end{equation}

Suppose that $\alpha$ is an $I(m) \oplus \pi$-marked skew decomposable permutation. Any skew decomposable permutation can be decomposed uniquely into skew indecomposable parts (see Figures~\ref{fig:factorizationA} and \ref{fig:factorizationB}). Since the pattern $I(m) \oplus \pi$ is sum decomposable, if a skew decomposable permutation is $I(m) \oplus \pi$ marked, then all the marks must occur within a single one of these components (Figure~\ref{fig:factorizationA}). Such components are enumerated by $P_{I(m) \oplus \pi}^\oplus$. The remainder of the permutation is therefore unmarked. The portions of the permutation to the left and to the right of the marked component are arbitrary, and are each counted by $S$. If the permutation itself is skew decomposable, either the left or the right part must be nonempty. This leads to the following relationships

\begin{equation}\label{eqn:factorization3}
      P_{I(m) \oplus \pi}^\ominus = 
        (S^2-1) P_{I(m) \oplus \pi}^\oplus.
\end{equation}

The case for marked instances of $I(m) \oplus \pi$ in sum decomposable permutations is more complicated. As above, suppose that $\alpha$ is an $I(m) \oplus \pi$-marked sum decomposable permutation. Such a permutation can be decomposed into sum indecomposable parts, and marks can occur in multiple such parts (Figure~\ref{fig:factorizationB}). Consider the components from left to right. There may be an arbitrary number of unmarked parts before the first marked one; this segment is counted by $S$.  The first marked part may have $i$ marked entries, where $1 \leq i \leq m$, or might contain all the marked entries (since $\pi$ is sum-indecomposable, if any mark of $\pi$ occurs in a sum indecomposable part, then all the marks must). The remainder of the permutation is then counted by $P_{I(m-i) \oplus \pi}$, in the first case, or $S$ in the second case. This leads to the following relationship:

    \begin{equation}\label{eqn:factorization4}
      P_{I(m) \oplus \pi}^\oplus = (S^2 - 1) P_{I(m) \oplus \pi}^\ominus
        + S \sum_{i = 1}^m P_{I(i)}^{\ominus, \bullet} 
        P_{I(m-i) \oplus \pi}.
    \end{equation}

    By our inductive hypothesis, we have that $P_{I(m-i) \oplus \pi} = F_{m-i} P_\pi$. Combining equations~\ref{eqn:factorization3} and \ref{eqn:factorization4} yields
    
    \begin{equation}\label{eqn:factorization5}
      P_{I(m) \oplus \pi}^\oplus = (S^2 - 1)^2 P_{I(m) \oplus \pi}^\oplus
        + S P_\pi \sum_{i = 1}^m  P_{I(i)}^{\ominus, \bullet} 
        F_{m-i}.
    \end{equation}
      
    Rearranging and combining with equation~\ref{eqn:factorization3} gives

    \begin{equation}\label{eqn:factorization6}
      P_{I(m) \oplus \pi} = \left(
      \frac{S^3 \sum_{i = 1}^m  P_{I(i)}^{\ominus, \bullet} 
      F_{m-i}}{1 - (S^2-1)^2} \right) P_\pi.
    \end{equation}

    Noting that the first factor in equation~\ref{eqn:factorization6} is independent of $\pi$ completes the proof. 
  \end{proof}
  
  \begin{figure}
		\minipage{0.45\textwidth}
			\begin{center}
				\begin{tikzpicture}[scale=.1]
                                        \foreach \i in {0,5,10,20}
                                          \draw[rounded corners = 1mm] (\i,-\i) rectangle (\i+5, -\i-5);
                                        \draw[thick, dotted] (16,-16) -- (19,-19);
                                    
                                      \end{tikzpicture}
			\end{center}
			\caption{A sum decomposable pattern must lie entirely within a single component of a skew decomposable permutation. }
			\label{fig:factorizationA}
		\endminipage\hfill
		\minipage{0.45\textwidth}
			\begin{center}
				  \begin{tikzpicture}[scale=.1]
                                        \foreach \i in {0,5,10,20}
                                          \draw[rounded corners = 1mm] (\i,\i) rectangle (\i+5, \i+5);
                                        \draw[thick, dotted] (16,16) -- (19,19);
                                    
                                      \end{tikzpicture}
			\end{center}
			\caption{A sum decomposable pattern may be spread across several components of a sum decomposable permutation.}
			\label{fig:factorizationB}
		\endminipage
	\end{figure}

  The functions $F_{m}$ in Lemma~\ref{lem:factorization} can be found easily. Considering the pattern $I(m) \oplus 1$, the popularity generating function can be computed as $P_{I(m) \oplus 1} = F_m P_1$, or equivalently, 

\[
F_m = \frac{P_{I(m+1)}}{P_1}.
\]

In particular, for the wedge permutation $w(\lambda)$ with $\lambda = \lambda_1, \lambda_2, \dots, \lambda_k$, we can apply Lemma~\ref{lem:factorization} iteratively (using the fact that complements are equipopular) to obtain

\begin{equation}
\label{eq:wedge_popularity}
P_{w(\lambda)} = \frac{P_{I(\lambda_1 +1)} P_{I(\lambda_2 +1)}
    \cdots P_{I(\lambda_k +1)}}{P_1^{k-1}}. 
\end{equation}
  
  It follows that, given the popularity generating function for an unknown pattern, if we can identify the factors of the form $F_m$, we can identify the pattern. We investigate this in the next section. Note also that this means that Proposition \ref{prop:abexchange} is not actually needed to confirm equipopularity of two different decomposition trees having the same internal degrees -- but as noted previously it does provide a mechanism for creating explicit bijections between the marked permutations of two different types within an equipopularity class.

\subsection{Identifying the Partition}

This section focuses on the popularity generating functions for the monotone patterns. Our goal is to show that, given an arbitrary product of such functions, the individual factors can always be identified. Define a bivariate generating function

\[
P(u,t) = \sum_{n\geq 0} P_{I(n)}(t)u^n.
\]

Note that $P(u,t)$ can be thought of as the generating function for marked separable permutations where the marks form an increasing sequence, counted by the number of marks ($u$) and the length of the permutation ($t$). As such,  
we can explicitly compute it using the structure of the separable permutations: 

\begin{lemma}\label{lem:P}
The function $P(u,t)$ is given by
\begin{align*}
    P = 
            \Bigg(&\left( (u+1)t^2 - 3(u+2)t + 3\right)r 
            - (3u - 17)t - 3(2u + 3)t^2 \\ &
            + (u+1)t^3 
            + \left(r(t-3) - 6t + t^2 - 3\right)s + 3\Bigg)/
            (24t - 4t^2), \\
&\text{where} \quad s = 
        \sqrt{1 + (ur-3u-6)t + (u^2 + u + 1)t^2}, \\
&\text{and} \quad r = \sqrt{1 - 6t + t^2}. 
\end{align*}

  \end{lemma}
  \begin{proof}

Let $P^\oplus$ and $P^\ominus$ denote the bivariate generating functions corresponding to $P$  but restricted to the sum and skew decomposable separable permutations respectively. Since the empty permutation cannot be marked, and the permutation of length one can either carry a mark or not, and any longer separable permutation is either sum or skew decomposable, we have: 

\begin{equation}
\label{eqn:Pproof1}
      P = 1 + (u+1)t + P^\oplus + P^\ominus.
\end{equation}

A sum decomposable separable permutation can be expressed as the sum of a sum indecomposable separable and any other separable permutation (both of length at least one). Further, a marking of length $r$ in the first part and a marking of length $s$ in the second combine to form a marking of length $r+s$. Since the generating function for increasing markings within the sum indecomposable separable permutations is $P - P^\oplus - 1$, we have 

\begin{equation} 
\label{eqn:Pproof2}
P^\oplus = (P - P^\oplus - 1)(P - 1).
\end{equation}

In a skew decomposable separable permutation, the markings (if any) must lie entirely in one component (since they mark an increasing pattern).  Such a permutation must either be entirely free of marks, or consist of a marked skew indecomposable component together with a mark-free separable permutation on either side, at least one of which is nonempty. This leads to the following relationship. 

\begin{equation}\label{eqn:Pproof3}
P^\ominus = \frac{1}{2}(S - t - 1) + 
      \left(P - P^\ominus - \frac{1}{2} (S - t - 1) - 1\right) (S^2-1).
\end{equation}

where the first term accounts for the unmarked case, and the first factor in the second term accounts for non-skew decomposable separable permutations containing at least one mark.

Solving the equations~\ref{eqn:Pproof1}--\ref{eqn:Pproof3} yields the explicit generating function. 
\end{proof}
  
We can derive this bivariate generating function in a different way, which will yield a connection to the Gegenbauer polynomials, and will ultimately help to prove our main theorem. The Narayana numbers will be useful here. 

\begin{definition}
The \emph{Narayana numbers} are a two variable refinement of the Catalan numbers. The $(n,k)$th Narayana number is defined as 
\[
N_{n,k} = \frac{1}{n} \binom{n}{k} \binom{n}{k-1}.
\]
and conventionally $N_{0,0} = 1$. 
\end{definition}

The generating function for these numbers is given by 
\[ 
    N(u,t) = \sum_{n,k \geq 0} N_{n,k} t^n u^k = 
    \frac{1 - t - tu - \sqrt{(1 - t - tu)^2 - 4t^2u}}{2t}. 
\]

Now consider a sequence of polynomials $q_n$ defined as follows:
\[
q_n(x) = \sum_{k=0}^n N_{n,k} x^{k-1} (1-x)^{n-k}.
\]
Note that $q_0(x) = 1/x$. These polynomials are described in the OEIS~\cite{oeis} under sequence A174128 and can be expressed in terms of the ordinary hypergeometric function as
\begin{equation}
\label{eqn:q_hypergeometric}
    q_n(x) = \frac{(1-x)^n\ _2F_1\left(1-n, -n; 2; \frac{x}{1-x}\right)}{1-x}.
\end{equation}



\begin{lemma}
Let 
\[      
Q(u,t) = S + \sum_{n=1}^\infty 
      \frac{S^{3n-3} t^n\ q_{n-1}(S^{-2})}{(2-S^2)^{2n-1}}u^n.
\]
Then $P = Q$. Equivalently, for all $n \geq 1$,
\[
P_{I(n)}(t) =    \frac{S^{3n-3} t^n\ q_{n-1}\left(S^{-2}\right)}{(2-S^2)^{2n-1}}.
\]
\end{lemma}

  \begin{proof}

By the definition of the polynomials $q_n$ we have that 

\[
q_{n-1}(x) = \frac{[u^{n-1}] N\left(u(1-x), \frac{x}{1-x}\right)}{x}.
\]

This allows us to express $Q$ in terms of $N$:

\[
\begin{aligned}
    Q &= S + \sum_{n=1}^\infty \frac{s^{3n-3} t^n \ q_{n-1}(S^{-2})}
              {(2-S^2)^{2n-1}}u^n \\
      &= S + u\sum_{n=1}^\infty \frac{S^{3n-1} t^n 
          \left( [u^{n-1}] \left(N\left(u(1-S^{-2}), \frac{1}{S^2-1}
          \right)\right) \right)}{(2-S^2)^{2n-1}} u^{n-1} \\
      &= S + uN\left(\left( \frac{S^3tu}{(2-S^2)^2}\right)
              \left(1 - S^{-2}\right), \frac{1}{S^2-1}\right) \cdot
              \left(\frac{tS^2}{2-S^2}\right) . 
\end{aligned}
\]

After some prodding and coaxing, using the explicit formulas for $N$ and $S$, a computer algebra system will now verify that this expression matches with the result in Lemma~\ref{lem:P}.  Therefore $P = Q$.
  \end{proof}

  We now turn our attention to the polynomials $q_n$. We show that these polynomials are related to a family of orthogonal polynomials, which will allow us to uniquely identify the factors in an arbitrary product. 

  \begin{definition}
    The \emph{Gegenbauer polynomials} $C_n^{(\alpha)}(x)$ are a family of orthogonal polynomials on the interval $[-1,1]$ with respect to the weight function $(1-x^2)^{\alpha-1/2}$. They can be defined in terms of their generating function:

\[
\frac{1}{(1-xt+t^2)^\alpha} = \sum_{n \geq 0} C_n^{(\alpha)}(x)t^n.
\]

They can also be defined by the following recurrence:
\[
\begin{aligned}
    C_0^{(\alpha)}(x) &= 1 \\
    C_1^{(\alpha)}(x) &= 2\alpha x \\
    C_n^{(\alpha)}(x) &= \frac{1}{n} \left(2x(n+\alpha - 1) C_{n-1}^{(\alpha)}(x)
      - (n+2\alpha - 2) C_{n-2}^{(\alpha)}(x) \right).
\end{aligned}
\]
\end{definition}

\begin{lemma}
The polynomials $q_n$ satisfy
\[
q_n(x) = \frac{2(1-2x)^{n-1} C_{n-1}^{(3/2)}\left(\frac{x}{1-2x}\right)}
                    {n(n+1)}. 
\]
\end{lemma}
  \begin{proof}
    This follows from known representations of the Gegenbauer polynomials in terms of hypergeometric functions and Equation \ref{eqn:q_hypergeometric}, and can be verified by a computer algebra system (Maple, in this case). 
  \end{proof}

Since the Gegenbauer polynomials are orthogonal, it follows that given an arbitrary product of polynomials from the set $\{q_n(x)\}_{n\geq 0}$, we can always recover the factors. Combined with Lemma~\ref{lem:factorization}, Equation \ref{eq:wedge_popularity}, and suitable rescalings and changes of variable, this implies that given any popularity generating function of a wedge pattern, we can always recover the lengths of the monotone segments that make up the wedge. Therefore, two wedge patterns have the same popularity generating function if and only if they correspond to the same partition. We summarize in the following theorem.

\begin{theorem}\label{thm:main-result}
Two separable patterns are equipopular in the separable permutations if and only if they have the same signature. Therefore, the set of partitions of the integer $n-1$ is in bijection with the equipopularity classes for separable patterns of length $n$ in the separable permutations. 
\end{theorem}

\section{Summary and Discussion}

As we have seen, the equipopularity classes among separable permutations of length $n$ are in one-to=one correspondence with partitions of $n-1$. This result exactly parallels the corresponding results for the class $\operatorname{Av}(132)$ of 132-avoiding permutations obtained in \cite{chua:popularity, rudolph:popularity}. Although $\operatorname{Av}(132)$ is a subclass of the separable permutations, neither result immediately implies the other (since the popularities are computed with respect to different universes). However, our techniques can be adapted to $\operatorname{Av}(132)$ to obtain new, and arguably more uniform, proofs of those already known results.

There is a wider collection of classes whose permutations correspond to more general decomposition trees -- these are the \emph{substitution closed classes}. Aside from trivial instances, the separable permutations are the smallest of such classes. In the particular case where a substitution closed class has only finitely many simple permutations (see \cite{albert:simple-permutat:} for background material) it seems likely that analogous arguments to those used here could be applied to define invariants which would give upper bounds on the number of equipopularity classes. However, it seems much less likely that the technical and algebraic arguments used here in order to obtain matching lower bounds (i.e., in determining that all classes with differing invariants have differing popularity) would generalize to that broader context.

Concerning those arguments, it is natural to wonder if there is a combinatorial explanation of the appearance of the Narayana numbers and the Gegenbauer polynomials in the popularity generating functions for monotone patterns. Of course the Narayana numbers are close relatives of the Catalan numbers, which in turn belong to the same family as the Schr\"{o}der numbers so this may be another instance where common sequences arise just because we are looking at simple problems (see \cite{zeilberger:opinion49}).

\end{document}